\documentclass[english,colorinlistoftodos]{scrartcl}

\usepackage{graphicx} 
\graphicspath{ {images/} }
\usepackage[margin=10pt, font=small, labelfont=it]{caption}

\usepackage{comment} 

\usepackage{amsmath, amsthm, amssymb, mathtools}

\usepackage[utf8]{inputenc}

\usepackage{xcolor}

\usepackage[inline]{enumitem}

\usepackage{todonotes}

\usepackage{hyperref} 

\hypersetup{
colorlinks=true,
urlcolor=cyan,	
filecolor=red,  	
linkcolor=blue 	
}

\usepackage{cleveref}[capitalise]	

 \newtheorem{theorem}{Theorem}

\newtheorem{lemma}[theorem]{Lemma}

\theoremstyle{definition}

\theoremstyle{remark}

\newcommand\cD{\mathcal{D}}

\newcommand\cP{\mathcal{P}}

\newcommand\bE{\mathbb{E}}

\newcommand\bN{\mathbb{N}}

\newcommand\bR{\mathbb{R}}

\newcommand{\spt}[1]{\text{spt}\left({#1}\right)}
\DeclareMathOperator{\dist}{\text{dist}}

\DeclareMathOperator{\leqc}{\leq_c}

\newcommand{\cco}[1]{\overline{\text{co}}\left({#1}\right)}

\newcommand{\ri}[1]{\text{ri}\left({#1}\right)}
\newcommand{\inside}[1]{\text{int}\left({#1}\right)}

\DeclareMathOperator{\MT}{\mathsf{MT}}
\DeclareMathOperator{\cpl}{\mathsf{Cpl}}

\DeclareMathOperator{\dopt}{\psi_{\text{lim}}}
\DeclareMathOperator{\Caff}{C_{\mathrm{q}}^{\text{aff}}\left(\mathbb{R}^{d}\right)}
\DeclareMathOperator{\bary}{\text{bar}}
\DeclareMathOperator{\mcov}{\text{MCov}}
\DeclareMathOperator{\sm}{\mathcal{M}_2}
\newcommand\wass{\mathcal{W}_2}

\usepackage{authblk}
\author[1]{Walter Schachermayer \thanks{\href{mailto:walter.schachermayer@univie.ac.at}{walter.schachermayer@univie.ac.at}}}
\affil[1]{Department of Mathematics, University of Vienna}
\author[2]{Pietro Siorpaes 
\thanks{\href{mailto:p.siorpaes@imperial.ac.uk}{p.siorpaes@imperial.ac.uk}}}
\affil[2]{Department of Mathematics, Imperial College London}

\usepackage{fancyhdr}
\title{Stretched Brownian Motion: convergence of dual optimising sequences 
	\thanks{This research was funded by the Austrian Science Fund (FWF) [Grant DOIs: 10.55776/P35197 and 10.55776/P35519).
For open access purposes, the authors have applied a CC BY public copyright license to any author accepted
manuscript version arising from this submission.
			}
		}
\date{\today. }

\begin{document}

\maketitle

 \begin{abstract}
We consider an irreducible pair $\mu \leqc \nu$ of probability measures on $\mathbb{R}^d$ in convex order. In \cite{BBST23},  Backhoff, Beiglb\"ock, Schachermayer and Tschiderer 	have shown that the Stretched Brownian Motion  from $\mu$ to $\nu$ is a Bass martingale, that there exists a dual optimiser $\dopt$, and the following somewhat surprising convergence result:  
by adding affine functions, one can make  any dual optimising sequence $(\psi_n)_n$ (satisfying some minor technical conditions)
converge pointwise to $\dopt$, save possibly on the relative boundary of the convex hull 
of the support of $\nu$. In the present paper we deal with the more delicate issue of convergence on said boundary, showing in particular that $\dopt$ is $\nu$ a.s.~finite, and $(\psi_n)_n$ converges to $\dopt$ in $\nu$-measure.

\bigskip


 \end{abstract}

\section{Introduction}
In mathematical finance one can construct several arbitrage-free models which are compatible with  observed market prices of vanilla options on the spot price $S=(S_t)_{t\in [0,T]}$.  In practice, only options with some possible maturities $0=T_0< T_1 <  \ldots <  T_n=T$ are traded, whereas to calibrate most models one would need  vanilla options prices across the whole continuum of times $t\in [0,T]$. This has traditionally been dealt with via a time-interpolation of the volatility at the unobserved maturities, which can however introduce instabilities. 
 
A possible and quite recent solution  is to consider models which are instead calibrated  to discrete marginals, such as the local variance Gamma model \cite{CaNa17},  the martingale Schr\"odinger bridge \cite{Hela19MSB}, and the Bass local volatility model which, in the continuous-time limit, converges to the well-known Dupire local volatility model \cite{Du97}. 
 
 The Bass local volatility model $S$ arises in a very natural way, as $S$ is the continuous process such that  $(S_t)_{T_i\leq t \leq T_{i+1}}$ is the martingale diffusion which best interpolates between the (known!) marginals $\mu_i:=Law(S_{T_i})$ and $\mu_{i+1}:=Law(S_{T_{i+1}})$, in the sense that it is the one which is as close as possible to Brownian Motion \cite{BaBeHuKa20}. Since we only need to restrict our attention to the generic interval $[T_i, T_{i+1}]$, we assume w.l.o.g.~that $T_i=0, T_{i+1}=1$, and write $\mu,\nu$ for $\mu_i,\mu_{i+1}$. 

 The corresponding optimisation problem,   whose solution $S$ has been called Stretched Brownian Motion (SBM), is  the Martingale Benamou-Brenier (MBB) problem, which was introduced in  \cite{BaBeHuKa20} using a probabilistic approach and in \cite{HuTr19} using PDEs, and  had already appeared independently in \cite{Loe18Opt}  in the context of market impact in finance. SBM has been studied in \cite{BBST23,BaScTs23,ScTs24}. If  $(\mu,\nu)$ is irreducible the corresponding SBM $S$ is a Bass martingale \cite[Theorem 1.3]{BBST23}, and thus admits an explicit construction in terms of a probability $\alpha$ on $\bR^d$, and a convex function $\psi: \bR^d \to (-\infty,\infty]$.
 To calibrate the model one needs to compute such $\alpha$, which is the solution to a fixed-point equation and is the minimiser of the so-called Bass functional, and thus  can be computed via  a fixed-point iteration scheme and be identified via the gradient descent for the Bass functional, see \cite{CoHe21,AcMaPa23,JoLoOb23,QiChYaFe24,BaScTs23,BaPaSc24}.

\vspace{0.3cm}
Since the SBM is defined as the solution to a convex optimisation problem, its study is carried out also by considering the corresponding dual optimisation problem. 
The main result of this paper strengthens the existing results about the convergence of  optimising sequences for such dual optimisation problem, whose solution $\psi$ (which exists if $(\mu,\nu)$ is irreducible) is the convex function mentioned above. 

Before stating our main result, we now introduce some notations and definitions.
Let $\cP(\bR^{d})$ be the space of Borel probabilities on $\bR^d$,  $\cP_{p}(\bR^{d})$ be its subspace of probabilities with finite $p^{th}$ moment, and for $p\geq 1$ let $\cP_{p}^x(\bR^{d})$ be the subspace of all the $\beta \in \cP_{p}(\bR^{d})$ whose barycentre $\bary(\beta):=\int y \beta(dy)$ equals $x\in \bR^d$. We assume that $\mu,\nu\in \cP_2(\mathbb{R}^d)$, and  denote by $\cpl(\mu, \nu)$ the set of transports from $\mu$ to $\nu$, i.e., the set  of probabilities $\pi$ on $\mathbb{R}^d \times \mathbb{R}^d$ with marginals $\mu$ and $\nu$. 
Each coupling $\pi \in \cpl(\mu, \nu)$ can be disintegrated with respect to $\mu$, i.e.~there exists a ($\mu$ a.s.~unique)  kernel $(\pi_x)_x$ such that $\pi( dx, dy)=  \mu(d x)\pi_{x}(dy)$, called the $\mu$-disintegration of $\pi$; clearly $\nu \in \cP_2(\bR^d)$ implies $\pi_x\in \cP_2(\bR^d)$   for $\mu$ a.e.~$x$.
We call $\pi$ a martingale transport, and $(\pi_x)_x$ a martingale kernel, if $\pi_x \in \mathcal{P}_2^x (\mathbb{R}^d)$ holds $\mu(dx)$-a.e.; we denote with  $MT(\mu,\nu)$ the set of all  martingale transports in $\cpl(\mu, \nu)$.
It follows from no-arbitrage arguments that $\mu, \nu$ are in convex order $\mu\leqc \nu$, or equivalently (by Strassen's theorem) that there exists a martingale transport $\pi \in \MT(\mu, \nu)$.

By definition the Stretched Brownian Motion $M^*$ between $\mu$ and $\nu$ is the unique optimiser to the  continuous-time optimisation problem 
\begin{align}
\label{eq: def SBM as maximiser}
\inf _{\substack{M_0 \sim \mu, M_1 \sim v, M_t=M_0+\int_0^t \sigma_s d B_s}} \mathbb{E}\left[\int_0^1\left|\sigma_t-\mathrm{Id}\right|_{\mathrm{HS}}^2 d t\right],
\end{align}
where $B$ is a Brownian motion on $\mathbb{R}^d$ and $|\cdot|_{\text {HS }}$ denotes the Hilbert-Schmidt norm.
It turns out that problem \eqref{eq: def SBM as maximiser} is equivalent to the  discrete-time optimisation problem 
\begin{align}
\label{eq: primal problem}
\sup_{\pi \in \MT(\mu,\nu)}\int \mcov(\pi_x,\gamma) \mu(dx),
\end{align}
where $\gamma$ denotes the standard Gaussian law on $\mathbb{R}^d$ and 
$$
\mcov\left(p_1, p_2\right):=\sup_{q \in \cpl(p_1, p_2)}  \int\left\langle x_1, x_2\right\rangle q\left(d x_1, d x_2\right), \quad p_1, p_2 \in \cP_2(\bR^d)
$$
is the maximal covariance between $p_1$ and $p_2$. Indeed, the unique optimiser $\pi^{SBM}$ of \cref{eq: primal problem} is closely related to $M^*$,  and so are the corresponding optimal values; so, $\pi^{SBM}$ is  also called the Stretched Brownian Motion between $\mu$ and $\nu$. 
To study problems \eqref{eq: def SBM as maximiser},\eqref{eq: primal problem} it is  useful to consider the dual optimisation problem
\begin{align}
\label{eq: dual opt problem}
D(\mu,\nu):= \inf_{\substack{\mu(\psi<\infty)=1 \\ \psi \text { convex} }}\cD(\psi),
\end{align}
where  given some $\pi \in \MT(\mu, \nu)$  the functional 
\begin{align}
\label{eq: def D}
\cD(\psi):=\int\left(\int \psi(y) \pi_{x}(d y)-\varphi^{\psi}(x)\right) \mu(d x),
\end{align}
is  defined for
 $\psi: \bR^d \to (-\infty,\infty]$ convex and $\mu$ a.s.~finite via the auxiliary function
\begin{align}
\label{eq: def varphi}
\varphi^{\psi}(x):=\inf _{p \in \mathcal{P}_{2}^{x}\left(\mathbb{R}^{d}\right)} \left(\int \psi d p - \mcov(p, \gamma)\right).
\end{align}
One can check that $\cD(\psi+a)=\cD(\psi)$ if $a$ is affine. By   taking $p=\delta_x$ in \cref{eq: def varphi} and applying Jensen's inequality, we get
\begin{align}
\label{eq: varphi psi and int psi}
-\infty \leq \varphi^{\psi}(x)\leq \psi(x) \in \bR \quad \text{and} \quad \bR \ni \psi(x) \leq \int \psi(y) \pi_{x}(d y)\leq \infty  \text{ for $\mu$ a.e.~} x, 
\end{align}
which shows that $\int \psi d\pi_x -\phi^\psi(x)$ is well defined  and belongs to $[0,\infty]$ for $\mu$ a.e.~$x$, and thus $\cD(\psi)$ is well defined and $\cD(\psi)\in [0,\infty]$, and if $\cD(\psi)<\infty$ holds then
\begin{align}
\label{eq: varphi and int psi finite}
-\infty < \varphi^{\psi}(x)  \leq \int \psi d \pi_x<\infty  \text{ for $\mu$ a.e.~} x. 
\end{align}

We recall  that  $D(\mu,\nu)$ does not depend on the choice of $\pi \in \MT(\mu, \nu)$ used in the definition of $\cD$ 
and  \cite[Theorem 3.3, Lemma 3.7]{BBST23}
\begin{align}
\label{eq: no duality gap}
\sup_{\pi \in \MT(\mu,\nu)}\int \mcov(\pi_x,\gamma) \mu(dx) = \int  \mcov(\pi^{SBM}_x, \gamma) \mu(d x) =D(\mu,\nu)
 \in \bR .
\end{align}
Denote by $C:=\cco{\spt{\nu}}$ the closed convex hull of the support $\spt{\nu}$ of $\nu$, and by $I=\ri{C}$ its  relative interior. We define $(\mu, \nu)$ to be \emph{irreducible}
\cite[Def.~1.2]{BBST23} if for any Borel sets $A, B$ such that $ \mu(A) > 0, \nu (B) > 0,$ there is $\pi \in \MT(\mu,\nu)$ with $\pi [A \times B] >0$; intuitively,  this means that $\pi$ transports positive mass from $A$ to $B$. 
In the rest of this section we assume that  $(\mu, \nu)$ is irreducible. In particular this implies that $\pi^{SBM}$ is a Bass martingale \cite[Theorem 1.3]{BBST23}, the dual problem \eqref{eq: dual opt problem} admits a lower semicontinuous solution $\dopt$   which satisfies $\mu(\ri{\psi<\infty})=1$ \cite[Theorem 7.6]{BBST23}, and this is unique modulo  affine functions \cite[Definition 7.14, Lemma 7.19]{BBST23}.
Moreover, there exists \cite[Theorem 7.8 and Propositions 7.13 and 7.20]{BBST23} a dual optimising sequence $\psi_{n}\geq 0, n\in \bN$ such that $\sup _{n} \psi_n<+\infty$ on $I$ and which belongs to the  space  $\Caff$ of continuous test functions which satisfy a convenient quadratic growth condition defined in \cite[Eq.~(2.1)]{BBST23}, and surprisingly for any such $(\psi_{n})_n$ there exists affine functions $(a_n)_n$ such that the sequence $(\psi_{n}+a_n)_n$ (which is also dual optimising) converges pointwise in $I \cup C^c$ to $\dopt \geq 0$, and $I \subseteq \{\dopt<\infty\} \subseteq C$.
The fact that $\left(\psi_{n}\right)_{n}$ is  dual optimising and $\dopt$ a dual optimiser means that
\begin{align}
\label{eq: dual is solvable}
\inf_{\substack{\mu(\psi<\infty)=1 \\ \psi \text { convex} }} \cD(\psi)=\lim_n \cD(\psi_n) =\cD(\dopt).
\end{align}

We can now state our main result, which refines the above convergence results by considering  the behaviour of   $\left(\psi_{n}\right)_{n}$  also on the relative boundary of $C$, on which $\nu$ may very well put strictly positive mass.
  We denote with $L^0(\nu)$ the space of $\nu$-equivalence classes of \emph{real-valued} functions on $\bR^d$, equipped with the convergence in $\nu$-measure. 
Given $V\subseteq \bR^d$, the notation $K \Subset V$ means that $K$ is a compact subset of $V$. 

\begin{theorem}
\label{prop: dual opt conv in proba}
Given an  irreducible pair $\mu \leqc \nu$ in $\cP_{2}(\bR^{d})$ define $C:=\cco{\spt{\nu}}$ and $I:=\ri{C}$. 
Let $(\psi_n)_n$ be a dual optimising sequence and $\dopt$ a dual optimiser, i.e.~$\psi_n,\dopt: \bR^d \to (-\infty,\infty], n\in \bN,$ are convex and $\mu$ a.s.~finite and such that \cref{eq: dual is solvable} holds. Assume w.l.o.g.~that $(\psi_n)_n$ are positive and converges pointwise on $I \cup C^c$ to  $\dopt \geq 0$, and $\dopt$ is lower semicontinuous and satisfies $\mu(\ri{\dopt<\infty})=1$.
Then $\dopt \in L^0(\nu)$ (i.e.~$\dopt<\infty$, $\nu$ a.s.), $\left(\psi_{n}\right)_{n}$ converges to  $\dopt$ in $L^0(\nu)$, and
\begin{align}
\label{eq: liminf psi_n >= dopt}
\liminf _{n \rightarrow \infty} \psi_{n}(y) \geq \dopt(y) \quad \text{ for all } y \in \bR^d.
\end{align}
Moreover, if $\spt{\mu}\Subset I$ then  $\dopt \in L^1(\nu)$ and $(\psi_n)_n$ converges to $\dopt$ in $L^1(\nu)$.
\end{theorem}

\section{Proof of Theorem \ref{prop: dual opt conv in proba}}
In this section we state some auxiliary results and prove \cref{prop: dual opt conv in proba}. To prove the  $L^1(\nu)$ convergence in \cref{prop: dual opt conv in proba} we will need the following approximation lemma, which will allow us to replace a kernel concentrated on $C=\bar{I}$ by a kernel concentrated on some $K \Subset I$; note that we know that $(\psi_n)_n$ is converging uniformly on any $K \Subset I$. 
\begin{lemma}
\label{le: approx lemma}
Let $\mu,\nu,C,I$ be as in \cref{prop: dual opt conv in proba}, and $\pi \in \MT(\mu, \nu)$ have $\mu$-disintegration $(\pi_x)_{x\in I}$, so that  $\pi_x\in \cP_2^x(\bR^d)$ and  $\spt{\pi_x}\subseteq C$ for $\mu$ a.e.. Then there exists an increasing sequence of compact convex sets $(K^j)_j \subseteq I$  and a sequence of kernels $((\pi^j_x)_{x\in I})_j$ such that $\cup_j K^j=I$, $K^j\subseteq \ri{K^{j+1}}$, $\pi^j_x\in  \cP_2^x(\bR^d)$ for $j\in \bN$ and $\mu$ a.e.~$x \in I$, and:
\begin{enumerate}
\item \label{it: cpt support}
$\spt{\pi^j_x}\subseteq K^j\,\,$   for $\mu$ a.e.~$x\in K^j$, and $\pi^j_x=\delta_x$ for $\mu$ a.e.~$x \in I \setminus K^j$, for  all $j\in \bN$,
\item $\pi^{j}_x \leqc \pi^{j+1}_x  \leqc \pi_x$  for $j\in \bN$ and $\mu$ a.e.~$x\in I$,
\item For $\mu$ a.e.~$x\in I$, as $j\to \infty$ we have $\wass(\pi^j_x, \pi_x)\to 0$ and
\begin{align}
\label{eq: mcov converge}
	0\leq \mcov\left(\pi^j_x, \gamma\right) \uparrow \mcov\left(\pi_x, \gamma\right)\leq \int \frac{1}{2}\|y\|^2 (\pi_x+\gamma)(dy) \in L^1(\mu(dx)).
\end{align}
\end{enumerate} 
\end{lemma}
To prove the convergence in measure in \cref{prop: dual opt conv in proba} we will use the statement about $L^1(\nu)$ convergence in \cref{prop: dual opt conv in proba}, plus the localisation procedure described in the next lemma, which is of independent interest; note that an analogous statement holds (with analogous proof) if $\mu^B$ is replaced by any probability $\mu'\ll \mu$ with bounded density $\frac{d \mu'}{d\mu}$.
\begin{lemma}
\label{prop: SBM from B}
Given  $\mu \leqc \nu$ in $\cP_{2}(\bR^{d})$, let $\pi^{SBM}$ be the Stretched Brownian Motion between $\mu$ and $\nu$ and  $(\pi_x^{SBM})_x$ be its $\mu$-disintegration. Given a Borel  set $B\subseteq \bR^d$ with $\mu(B)>0$, define 
\begin{align}
\label{}
\mu^B:=\frac{\mu(B \cap \cdot)}{\mu(B)},  \quad \nu^B:=\int \mu^B(dx) \pi_x^{SBM} , \quad \pi^B(dx,dy):=\mu^B(dx) \pi_x^{SBM}(dy).
\end{align}
  Then the Stretched Brownian Motion between $\mu^B$ and $\nu^B$ is $\pi^B$. Let $\psi_n, n\in \bN$ be convex and $\mu$ a.s.~finite; if  $(\psi_n)_{n\in \bN}$ is a dual optimising sequence for $(\mu,\nu)$ then it is a dual optimising sequence for  $(\mu^B,\nu^B)$. Moreover, if $(\mu,\nu)$ are irreducible then so are $(\mu^B,\nu^B)$ (equivalently, if $\pi^{SBM}$ is a Bass martingale then so is $\pi^B$), $\nu_B\sim \nu$ holds (so in particular $\cco{\spt{\nu_B}}=\cco{\spt{\nu}}$), and the  dual optimiser  $\dopt$ for $(\mu,\nu)$  is a dual optimiser for $(\mu^B,\nu^B)$.
\end{lemma}

To combine \cref{le: approx lemma,prop: SBM from B} we will need the following  result.
\begin{lemma}
\label{le: limit in proba on subsets}
Let $\mu\in \cP(\bR^{d})$, and $\mu(I^c)=0$ for some Borel set $I \subseteq \bR^d$. 
Let $g,g_n,h,h_n:I \to \bR$ be Borel functions, and $(I_j)_{j\in \bN}$ be Borel subsets of $\bR^d$ such that 
$\mu(I \setminus \cup_{j} I_j)=0$ and $\mu(I_j)>0$ for all $j$. Define $\mu_{j}:=\mu(I_j \cap \cdot)/\mu(I_j) \in \cP(\bR^{d})$. Then:
\begin{enumerate}
\item \label{it: mu}   $g_n\to g$ in $L^0(\mu)$ if and only if  $g_n\to g$ in $L^0(\mu_j)$ for all $j\in \bN$.
\item \label{it: nu} Given a martingale kernel  $(\pi_x)_x$, define  
$$\textstyle \nu:=\int \mu(dx) \pi_x, \quad \nu_j:=\int \mu_j(dx) \pi_x \quad \in \cP_{1}(\bR^{d}),$$  then  $h_n\to h$ in $L^0(\nu)$ if and only if  $h_n\to h$ in $L^0(\nu_j)$ for all $j\in \bN$.
\end{enumerate} 
\end{lemma}

The proof of \cref{le: approx lemma,prop: SBM from B,le: limit in proba on subsets} are postponed to section \ref{se: proofs of lemmas}.

\begin{proof}[Proof of \cref{prop: dual opt conv in proba}.]   We can assume w.l.o.g.~that the affine hull of the support of $\nu$ equals $\bR^d$ \cite[Assumption 3.1 and text that follows]{BBST23}, and so $I$ is open. 
Recall that if $\cD(\psi)<\infty$ then \cref{eq: varphi and int psi finite} holds, and so $\varphi^{\psi},\int \psi d \pi_\cdot \in L^0(\mu)$; for this reason we assume that $\cD(\psi_n)<\infty$ for all $n$, which we can do w.l.o.g.~since $\cD(\psi_n)\to \cD(\dopt)<\infty$.

\medskip 

\emph{We now prove \cref{eq: liminf psi_n >= dopt}}. Assume by contradiction that \cref{eq: liminf psi_n >= dopt} fails at some point $y$; clearly $y \in C\setminus I$, since by assumption  $\left(\psi_{n}(y)\right)_{n}\to \dopt(y)$ for all $y \in I \cup C^c=(C\setminus I)^c$.  Choose a $x_0\in I$; by restricting all functions to the segment from $y$ to $x_0$ (i.e.~by replacing $\bR^d \ni z\mapsto f(z)$ with $[0,1] \ni t\mapsto f(tx_0+(1-t)y)$) we can assume w.l.o.g.~that $d=1, y=0,x_0=1,I=(0,a)$ for some $a>1$. 
 We focus on the case $\psi_{\textnormal{lim}}(0) < \infty$, and leave the case $\psi_{\textnormal{lim}}(0)=\infty$ to the reader.
 As we assumed that  \cref{eq: liminf psi_n >= dopt} fails at $y$, by passing to a  subsequence (without relabelling)  we get that $\epsilon:=\dopt(0)-\lim_{n \rightarrow \infty} \psi_{n}(0)>0$ and so 
for all big enough $n\in \bN$
\begin{align}
\label{eq: psi_n small in the sides}
 \psi_{n}(0)\leq \dopt(0)-\frac{3\epsilon}{4}.
\end{align}
 Since $\dopt$ is lower semicontinuous, we have $\dopt(x)\geq \dopt(0)-\frac{\epsilon}{4}$ for all $x\in (0,\delta)$ for some $\delta \in (0,1)$. 
Thus  for all $x\in (0,\delta)$ for all big enough $n\geq N(x)\in \bN$ we have
\begin{align}
\label{eq: psi_n big in the middle}
 \psi_{n}\left(x\right)\geq \dopt(0)-\frac{\epsilon}{2} .
\end{align}
If follows from \cref{eq: psi_n small in the sides,eq: psi_n big in the middle} that the slope $\frac{\psi_n(x)-\psi_n(0)}{x-0}$ of $\psi_n$ between $0$ and $x$ 
is bounded below by $\frac{\epsilon}{4x}$. Since $\psi_n$ is convex, such slope is bounded above by the left-derivative $\psi_n'(x-)$ of $\psi_n$ at $x$, and so $\frac{\epsilon}{4x} \leq \psi_n'(x-)$ for all $n\geq N(x)$. From this and \cref{eq: psi_n big in the middle}, using the convexity of $\psi_n$  we conclude that 
$$\liminf_n \psi_n(1) \geq \liminf_n  \psi_n(x) +\psi_n'(x-)(1-x)\geq \left(\dopt(0)-\frac{\epsilon}{2} \right)+ \frac{\epsilon}{4x}(1-x) .$$ 
Taking $\lim_{x \downarrow 0}$ gives  $\liminf_n \psi_{n}(1)= \infty$, contradicting $\psi_{n}(1)\to \dopt(1) \in \bR$. Thus, \cref{eq: liminf psi_n >= dopt} holds.

\medskip

\emph{We now prove that $\dopt \in L^1(\nu)$ if $\spt{\mu}\Subset I$.} Let $\pi^{SBM}$ be the Stretched Brownian Motion $\pi^{SBM}$ from $\mu$ to $\nu$ and  $(\pi_x^{SBM})_x$ be its disintegration with respect to $\mu$. By \cite[Lemma 7.9]{BBST23} and Fatou's lemma 
$$0 \leq A:=\int \int \left(\dopt(y)-\dopt(x)\right) \pi_x^{SBM}(d y) \mu(d x) <\infty,$$
i.e.~
\begin{align*}
\textstyle
  \int \int \dopt(y) \pi_x^{SBM}(d y) \mu(d x) \leq A+ \int \int \left(\dopt(x)\right) \pi_x^{SBM}(d y) \mu(d x),
\end{align*}
or equivalently $\int \dopt d\nu \leq  A+ \int \dopt d \mu$. So, from the fact that  $\dopt$ is continuous and finite on $I$, and thus bounded on the compact set $\spt{\mu} \subseteq I$, we conclude $\dopt \in L^1(\nu)$.

\medskip

\emph{We now prove that $(\psi_n)_n \to \dopt$ in $L^1(\nu)$ if $\spt{\mu}\Subset I$.}
Since  $\dopt, \psi_{n}\geq 0$, we get
$$0\leq (\psi_{n}- \dopt)^{-}\leq \dopt,$$
and as \cref{eq: liminf psi_n >= dopt} is equivalent to $(\psi_{n}- \dopt)^{-}(y)\to 0$ for all $y\in \bR^d$,  since $\dopt \in L^1(\nu)$ by dominated convergence we conclude that $(\psi_{n} - \dopt)^{-}\to 0$ in $L^1(\nu)$. Thus, to prove $(\psi_n)_n \to \dopt$ in $L^1(\nu)$
 it suffices to show that $\int \psi_{n} d \nu\to \int \dopt d \nu$. Since \cref{eq: liminf psi_n >= dopt} and Fatou's lemma imply that 
$\liminf_{n} \int \psi_n d\nu \geq \int \psi_{\textnormal{lim}} d\nu$, it suffices to show that
\begin{equation}\label{eq: limsup}
\limsup_{n} \int \psi_n d\nu \leq \int \psi_{\textnormal{lim}} d\nu.
\end{equation}
Since we assumed that $(\psi_n)_n$ is a dual optimising sequence, i.e.~ 
\begin{equation}\label{eq: optimising}
 \int \psi_n d\nu - \int \varphi^{\psi_n} d\mu = \cD(\psi_n) \to  \cD(\dopt)= \int \psi_{\textnormal{lim}} d\nu - \int \varphi^{\dopt} d \mu, 
\end{equation}
to prove \cref{eq: limsup} it remains to show that
\begin{equation}\label{eq: 2nd limsup}
\limsup_{n} \int\varphi^{\psi_n} d \mu \leq \int \varphi^{\psi_{\textnormal{lim}}} d \mu.
\end{equation}
For $\pi_x:=\pi_x^{SBM}$, let $K^j,\pi^j_x$ be as in \cref{le: approx lemma}. Since $\pi^j_x\in  \cP_2^x(\bR^d)$, the definition of $\varphi^{\psi}$ gives 
\begin{align}
\label{eq: psi_n bdd above}
\varphi^{\psi_n} \leq 	\int \psi_n d \pi^j_\cdot - \mcov(\pi^j_\cdot, \gamma) \quad  \mu \text{ a.e.}.
\end{align}
Since $(\psi_n)_n$ are convex and $\psi_n\to \psi<\infty$ on $I$, the convergence $\psi_n\to \psi$ is uniform on compacts \cite[Theorem 3.1.4]{HiLe01}. It follows from   \cref{it: cpt support} of \cref{le: approx lemma} that the support of $\nu^j:=\int\mu(dx) \pi^j_x$ satisfies $\spt{\nu^j} 
\subseteq K^j \cup \spt{\mu}$, and thus it is compact. 
Thus we get 
\begin{align} 
\label{eq: psi_n pi^j}
\lim_n \int \left( \int \psi_n d \pi^j_\cdot \right) d\mu =\lim_n \int \psi_n d \nu^j= \int \psi d\nu^j=\int \left( \int \psi d \pi^j_\cdot  \right) d\mu .
\end{align}
Since $\pi^{j}_x  \leqc \pi_x$ gives $\int \left(\int \psi d \pi^j_\cdot\right) d\mu  \leq \int \left( \int \psi d \pi_\cdot \right) d\mu$, integrating \cref{eq: psi_n bdd above}, taking $\limsup_n $ and using \cref{eq: psi_n pi^j} we get that
\begin{align} 
\label{eq: phipsi_n bdd by phipsi}
\limsup_n \int \varphi^{\psi_n} d\mu  \leq \int \left(	\int \psi d \pi_\cdot - \mcov(\pi^j_\cdot, \gamma)\right) d\mu .
\end{align}
By dominated convergence it follows from \cref{eq: mcov converge} that $\mcov\left(\pi^j_\cdot, \gamma\right) \to \mcov\left(\pi_\cdot, \gamma\right)$ in $L^1(\mu)$, 
and so taking $\lim_j$ of \cref{eq: phipsi_n bdd by phipsi} 
we conclude that \cref{eq: 2nd limsup} holds, since $\pi_x$ is the solution of the minimisation problem \eqref{eq: def varphi} for $\mu$ a.e.~$x\in I$ when $\psi=\dopt$ \cite[Section 5, sentence after eq.~(5.2)]{BBST23}. 
We have thus proved that $(\psi_n)_n \to \dopt \in L^1(\nu)$ if $\spt{\mu}$ is compact.

\medskip

\emph{We now prove that $\dopt \in L^0(\nu)$ and $(\psi_n)_n\to \psi_{\textnormal{lim}}$ in $L^0(\nu)$.}
For $\pi_x:=\pi_x^{SBM}$, let $K^j$ be as in \cref{le: approx lemma}, and define  
$$\mu_{j}:=\mu(\cdot| K^j):=\frac{\mu(K^j \cap \cdot)}{\mu(K^j)},  \quad \nu_j:=\int\mu_j(dx) \pi_x^{SBM} .$$
For any $j\in \bN$, by \cref{prop: SBM from B} $\dopt$ is a dual optimiser and $(\psi_n)_n$ is a dual optimising sequence  also for $(\mu_j,\nu_j)$. Moreover, $(\mu_j,\nu_j)$ satisfy  $\spt{\mu_j}=K^j\Subset I= \ri{\cco{\spt{\nu_j}}}$ and $\nu_j \sim \nu$. 
Thus from the previously proved statements we get that $\dopt \in L^1(\nu_j)$, and so $\dopt$ is finite $\nu_j$ a.s.~and so $\nu$ a.s. (i.e.~$\dopt \in L^0(\nu)$), and  $\psi_{n} \to \dopt $ in $L^1(\nu_j)$, and so $\psi_{n} \to \dopt $  in $L^0(\nu_j)$, and so by  \cref{le: limit in proba on subsets}  $\psi_{n} \to \dopt $  in $L^0(\nu)$.
\end{proof}

\section{Proofs of lemmas}
\label{se: proofs of lemmas}

In this section we first present the  proof of \cref{le: limit in proba on subsets}, and then prove \cref{le: approx lemma,prop: SBM from B} with the aid the additional lemmas \ref{prop: strassen extension} and \ref{prop: Mcov is increasing}.

\begin{proof}[Proof of \cref{le: limit in proba on subsets}]
 \Cref{it: mu} is well known. Let us prove  \cref{it: nu}. Let $\theta \in \cP(\bR^d)$ and $(\pi_x)_x$ be a martingale kernel,  and define  $\beta:=\int_{\bR^d}\theta(dx) \pi_x$. As $(h_n)_n\to h$ in $L^0(\theta)$ iff  $(1_{\{|h_n- h| > \epsilon\}})_n \to 0$ in $L^1(\theta)$ for all $\epsilon>0$, we get that  $(h_n)_n\to h$ in $L^0(\beta)$ iff,  for all $\epsilon>0$, the sequence  
 $$v^\epsilon_n:=\int_{\bR^d}\pi_\cdot(dy)1_{\{|h_n(y)- h(y)| > \epsilon\}}, \quad n\in \bN, $$
 converges to 0 in $L^1(\theta)$, or equivalently (since $|v_n^\epsilon|\leq 1$) in $L^0(\theta)$. Thus \cref{it: nu} follows from \cref{it: mu}  applying this fact to $\theta=\mu$ and then to $\theta=\mu_j$ with $g_n:=1_{\{|h_n- h| > \epsilon\}}, g=0$.
\end{proof}

\begin{lemma}
\label{prop: strassen extension}
If $\alpha,\beta, \zeta \in \cP_1(\bR^d)$, $\pi^1 \in \MT(\alpha,\beta)$, $\pi^2 \in \cpl(\alpha,\zeta)$,  then there exist random variables $A,B,Z$ such that $(A,B)\sim \pi^1,(A,Z) \sim \pi^2$ and $\bE[B|A,Z]=A$.
\end{lemma}
\begin{proof}
One can take $(A,B,Z)$ to be any random vector whose law $\pi$ is given by 
$$\pi(da,db,dz):=\alpha(da)\pi^1_a(db)\pi^2_a(dz), \quad \text{ where } \pi^i(da,dx)=\alpha(da)\pi^i_a(dx),$$
i.e.~the kernel $(\pi^i_a)_a$ denotes the disintegration of $\pi^i$ with respect to $\alpha$ for $i=1,2$. 
\end{proof} 
\begin{lemma}
\label{prop: Mcov is increasing}
If  $\alpha,\beta,\zeta \in \cP_2(\bR^d)$ and $\alpha \leqc \beta$ then 
$$\langle \bary(\alpha), \bary(\zeta) \rangle \leq\mcov\left(\alpha, \zeta\right) \leq \mcov\left(\beta, \zeta\right)\leq \int \frac{1}{2}\|x\|^2 (\beta+\zeta)(dx)<\infty .$$
\end{lemma} 
\begin{proof}
Fix any $\pi^2 \in \cpl(\alpha,\zeta)$. By Strassen's theorem $\exists \pi^1 \in \MT(\alpha,\beta)$, so applying \cref{prop: strassen extension} yields $(A,B,Z)$ such that $A \sim \alpha, B\sim \beta,Z\sim \zeta$ and
$$\textstyle 
 \bE\langle Z,B \rangle= \bE \left(\bE[\langle Z,B \rangle|A,Z]\right)=\bE\langle Z,\bE[B|A,Z] \rangle =\bE \langle Z,A \rangle =\int \pi^2(da,dz) \langle a,z \rangle,$$
 which, since $\pi^2 \in \cpl(\alpha,\zeta)$ was arbitrary, shows that $\alpha \leqc \beta$ implies  $\mcov\left(\alpha, \zeta\right) \leq \mcov\left(\beta, \zeta\right)=\mcov\left(\zeta,\beta\right)$, which applied to $\delta_{\bary(\alpha)}\leqc \alpha$ and then to $\delta_{\bary(\zeta)}\leqc \zeta$ gives
 $$\langle \bary(\alpha), \bary(\zeta) \rangle = \mcov\left(\delta_{\bary(\alpha)}, \delta_{\bary(\zeta)}\right) \leq \mcov\left(\alpha, \delta_{\bary(\zeta)}\right) \leq \mcov\left(\alpha, \zeta\right).$$
  Finally $\mcov\left(\beta, \zeta\right)\leq\frac{1}{2} \int \|x\|^2 (\beta+\zeta)(dx)$ follows from $2\langle x,y \rangle \leq \|x\|^2 +\|y\|^2$.
\end{proof}

\begin{proof}[Proof of \cref{le: approx lemma}]
As usual, by restricting to the affine hull of $C$, we may assume w.l.o.g. (see \cite[Assumption 3.1 and text that follows]{BBST23})~that $C$ has dimension $d$, i.e.~that $I$ is open.  
Let $K^j$ denote the intersection of the closed ball of radius $j$ with the set of $x \in C$ whose distance 
$\dist(x,I^c):=\inf\{\|x-b\|: b\in I^c\}$ from $I^c$ is at least $1/j$, i.e. 
\begin{align}
\label{eq: def of dist and L}
\textstyle 
K^j:=\left\{z\in C : \dist(z,I^c) \geq  \frac{1}{j} , \quad \|z \| \leq j\right\}.
\end{align}
Then $(K^j)_j \subseteq I$ is an increasing sequence of compact convex sets such that $\cup_j K^j =I$ and $K^j\subseteq \inside{K^{j+1}}$. For $x\in I$ let $(M^x_t)_{t\in [0,1]}$ be the Stretched Brownian motion between $\delta_x$ and $\pi_x$, and define the stopping times 
$$\tau_x^j:=\inf\left\{t\in [0,1]: M_t^x \notin K^j\right\} \wedge 1, \quad j\in \bN.$$
Note that, for $x \notin K^j$, we have $\tau^j_x =0$. Since  the limit $\tau_x$ of the increasing sequence $(\tau^j_x)_j$ equals 1 a.s.~for all $x\in I$  \cite[Corollary 6.8]{BBST23} and $M^x$ is continuous, we get that $M^x_{\tau^j}\to M^x_1$ a.s., and thus also in $L^2$ since Doob's $L^2$-inequality and $M^x_{1}\sim \pi_x \in \cP_2(\bR^d)$ imply $\sup_t \|M^x_t\|\in L^2$. 
Thus the law  $\pi^j_x$ of $M^x_{\tau^j}$ converges weakly to the law $\pi_x$ of $M^x_1$, and its second moment is finite and  also converges, i.e.~$\sm(\pi^j_x) \to \sm(\pi_x)$. It follows \cite[Theorem 7.12]{Vil03} that $\wass(\pi^j_x, \pi_x)\to 0$ as $j\to \infty$, and so the identity\footnote{This identity follows integrating the formula $\|x\|^2-2 \langle x,y \rangle + \|y\|^2=\|x-y\|^2$ w.r.t.~$r(dx,dy)$ and taking the infimum over $r\in \cpl(p,q)$.} 
$$\sm(p)-2\mcov(p,q)+\sm(q)= \wass(p,q), \quad p,q \in \cP_2(\bR^d)$$
implies $\mcov\left(\pi^j_x, \gamma\right)  \to\mcov\left(\pi_x, \gamma\right) $ as $j\to \infty$.
Since $M^x$ is a martingale we get $\pi^j_x \leqc \pi^{j+1}_x \leqc \pi_x$ for all $j$, so 
 \Cref{prop: Mcov is increasing} implies
\begin{align}
\label{eq: chain ineq mcov increasing}
0\leq \mcov\left(\pi^j_x, \gamma\right) \leq \mcov\left(\pi^{j+1}_x, \gamma\right)  \leq \mcov\left(\pi_x, \gamma\right)<\infty .
\end{align}
Clearly $M_{\tau^j}^x$ has values in $K^j$ for $x\in K^j$,  and  $M_{\tau^j}^x=x$ for $x \in I \setminus K^j$, and so $\pi^j_x$ is supported in $K^j$ (resp.~$\{x\}$) for $x\in K^j$ (resp.~$x \in I \setminus K^j$).
Finally, \cref{prop: Mcov is increasing} gives
 $$\mcov\left(\pi_x, \gamma\right)\leq \int \frac{1}{2}\|y\|^2 (\pi_x+\gamma)(dy)=:g(x),$$  and since $\nu=\int\mu(dx) \pi_x \in \cP_2(\bR^d) \ni \gamma$ it follows  that  $g\in L^1(\mu)$.
\end{proof}

\begin{lemma}
\label{prop: all depends just on the kernel}
Given  $\mu \leqc \nu$ in $\cP_{2}(\bR^{d})$, let $\pi^{SBM}$ be the Stretched Brownian Motion between $\mu$ and $\nu$ and  $(\pi_x^{SBM})_x$ be its $\mu$-disintegration.  
For any $\psi: \bR^d \to (-\infty,\infty]$ convex and $\mu$ a.s.~finite define 
$$L(\psi)(x):=\int \psi(y) \pi^{SBM}_x(dy) - \varphi^{\psi}(x), \quad x\in \bR^d.$$
Then 
\begin{align}
\label{eq: L bigger Mcov}	
L(\psi)\geq \mcov(\pi_\cdot^{SBM},\gamma) \geq 0 \quad \mu \text{ a.e.}
\end{align}
for any such $\psi$, and $(\psi_n)_{n}$ is dual optimising (i.e.~$\psi_n: \bR^d \to (-\infty,\infty]$ is convex and $\mu$ a.s.~finite for any $n\in \bN$, and satisfies \cref{eq: dual is solvable}) if and only if
\begin{align}
\label{eq: dual opt iff L1 norm to 0}
\| L(\psi_n) - \mcov(\pi_\cdot^{SBM},\gamma) \|_{L^1(\mu)} \to 0 \quad  \text{ as } \quad n \to \infty.
\end{align} 
\end{lemma}
\begin{proof}
The first inequality in \cref{eq: L bigger Mcov} follows  from  the definition  of $\varphi^{\psi}$ (\cref{eq: def varphi}), the second from \cref{prop: Mcov is increasing}.
By \cref{eq: no duality gap} $(\psi_n)_{n}$ is dual optimising if and only if
$$\cD(\psi_n)=\int L(\psi_n) d \mu \to \int \mcov(\pi_\cdot^{SBM},\gamma)d\mu <\infty$$ 
which by \cref{eq: L bigger Mcov} holds if and only if \cref{eq: dual opt iff L1 norm to 0} holds.
\end{proof}

\begin{proof}[Proof of \cref{prop: SBM from B}]
Trivially $\pi^B\in \MT(\mu^B,\nu^B)$.  
Choose any $\pi\in \MT(\mu^B,\nu^B)$, and let $(\pi_x)_x$ be its $\mu$-disintegration. Define 
\begin{align*}
\tilde{\pi}_x:= 
\begin{cases}
\pi_x & \text{ for } x\in B,\\
\pi^{SBM}_x & \text{ for } x\in \bR^d \setminus B.
\end{cases}, 
\quad \tilde{\pi}(dx,dy):=\mu(dx)\tilde{\pi}_x(dy).
\end{align*}

Notice that $\tilde{\pi}(\bR^d \times \cdot)=\nu$, and so $\tilde{\pi}\in \MT(\mu,\nu)$.  
The inequality
\begin{align}
\label{eq: pi^B maximal}
\int_{C} \mu(dx) \mcov(\pi_x^{SBM},\gamma) \geq \int_{C} \mu(dx) \mcov(\tilde{\pi}_x,\gamma) 
\end{align}
holds when $C=B$: indeed it holds when $C=\bR^d$ (since  $\tilde{\pi}\in \MT(\mu,\nu)$, this follows from \cref{eq: no duality gap}), and it holds with equality when $C=\bR^d \setminus B$ (by definition of $\tilde{\pi}_x$). Evaluating \cref{eq: pi^B maximal} with $C=B$ and dividing by $\mu(B)$ we get that
$$\int_{\bR^d} \mu^B(dx) \mcov(\pi_x^{SBM},\gamma) \geq \int_{\bR^d} \mu^B(dx) \mcov(\pi_x,\gamma)$$
and given that $\pi\in \MT(\mu^B,\nu^B)$ was arbitrary,  it follows that $\pi^B$ is the Stretched Brownian Motion between $\mu^B$ and $\nu^B$: indeed by  \cite[Theorem 3.3]{BBST23} the unique maximiser of $\pi \mapsto \int \mcov(\pi_x,\gamma) \mu'(dx)$ over $\pi \in \MT(\mu',\nu')$ is the Stretched Brownian Motion between $\mu'$ and $\nu'$, so the thesis follows taking $\mu'=\mu^B,\nu'=\nu^B$.

If \cref{eq: dual opt iff L1 norm to 0} holds then it holds with $\mu$ replaced by $\mu^B$, and since 
 the $\mu$-disintegration $(\pi^B_x)_x$ of the Stretched Brownian Motion  $\pi^B$ between $\mu^B$ and $\nu^B$ equals the $\mu$-disintegration  $(\pi^{SBM}_x)_x$ of the Stretched Brownian Motion  between $\mu$ and $\nu$,  if $(\psi_n)_{n}$ is a dual optimising sequence for $(\mu,\nu)$ then by \cref{prop: all depends just on the kernel}  it is also is a dual optimising sequence for $(\mu^B,\nu^B)$. 
 In particular, if a dual optimiser $\dopt$ for $(\mu,\nu)$ exists, then it is also a dual optimiser for $(\mu^B, \nu^B)$, since $\psi$ is a dual optimiser iff  $\psi_n=\psi, n\in \bN$ is a dual optimising sequence  (here an alternative proof: use $(\pi^B_x)_x=(\pi^{SBM}_x)_x$ and apply \cite[Definition 7.14, Lemma 7.19]{BBST23}).
 
By  \cite[Theorems 1.4 and Remark D.3]{BBST23} two probabilities are irreducible iff there exists a Bass martingale connecting them, and by \cite[Theorem 1.3]{BBST23} the Stretched Brownian Motion between irreducible measures is a Bass martingale, so the statements about irreducibility and Bass martingales are equivalent. 

To show that they hold, assume that $(\mu,\nu)$ is irreducible, so $\pi_x^{SBM} \sim \nu$  for $\mu$ a.e.~$x$ by \cite[Corollary 7.7]{BBST23}. Since, for any Borel  $A\subseteq \bR^d$, $\nu_B(A)=0$ holds iff $\pi_x^{SBM}(A)=0$ for $\mu_B$ a.e.~$x$, we conclude that $\nu_B(A)=0$  iff $\nu(A)=0$, i.e.~$\nu_B\sim \nu$; thus $C_B:=\cco{\spt{\nu_B}}$ equals $C:=\cco{\spt{\nu}}$, and  \cite[Theorem D.1]{BBST23} implies that $(\mu_B, \nu_B)$ is irreducible.  Finally, recall that if $(\mu,\nu)$ is irreducible,  the dual optimiser $\dopt$ exists \cite[Theorem 7.6]{BBST23}.
\end{proof}

\bibliography{librarySBM}{}
\bibliographystyle{alpha}

\end{document}